\documentclass[a4paper,12pt]{article}
\usepackage{amsmath,amssymb,amsthm,amsfonts}
\usepackage[utf8]{inputenc}
\usepackage[T1]{fontenc}
\usepackage{anysize}
\marginsize{2cm}{1.5cm}{2cm}{1.5cm}
\usepackage{mathptmx}
\usepackage{graphicx}

\title{Modified Special Functions Defined by Generalized M-Series and Their Properties}
\author{Enes Ata$^{1}$
\\ 
\small $^{1}$Department of Mathematics, Faculty of Arts Science,  Kırşehir Ahi Evran University, Kırşehir, Turkey.
\\
\small enesata.tr@gmail.com
\\
}
\date{}

\newtheorem{definition}{Definition}[section]
\newtheorem{theorem}{Theorem}[section]
\newtheorem{corollary}{Corollary}[section]

\begin{document}
\maketitle
\hspace{-0.6cm}\hrulefill
\begin{abstract}
In this paper, modified gamma and beta functions containing generalized M-series in their kernel are defined. Also, modified Gauss and confluent hypergeometric functions are defined using the modified beta function. Then, some properties of these modified special functions are obtained. Finally, the relationships between generalized special functions in the literature with newly defined modified special functions are examined.
\\

\end{abstract}
\hrulefill

\section{Introduction}
The theory of special functions is one of the important study areas of applied mathematics. Special functions are usually defined with the help of generalized integrals or infinite series. Some of the most important of these functions are gamma, beta, Gauss hypergeometric and confluent hypergeometric functions.

There have been many publications in recent years on various generalizations of special functions, (see for examples \cite{abubakar,ata20,ata21,chaudhry94,chaudhry97,chaudhry2004,choi,cetinkaya,goswami,kulip,lee,mubeen,ozergin,parmar,rahman,shadab,srivastava,sahin}). In many of these studies, generalizations are defined using more general functions instead of the term $\exp(-t)$ in the integral representation of the gamma function. Then, generalizations are given for the beta function by using a similar term without disturbing the symmetry properties. In addition, the generalized beta function is used to describe the generalizations of Gauss hypergeometric  and confluent hypergeometric functions.

The main aim of this paper is to expand the application areas of special functions. Therefore, inspired by these publications mentioned above, new generalizations for gamma, beta, Gauss hypergeometric and confluent hypergeometric functions are defined with the help of generalized M-series. In addition, Mellin transforms, Laplace transforms, beta transforms, integral representations, derivative formulas, transformation formulas and reduction relations of these defined functions have been obtained. Finally, the relationships between newly defined special functions and generalized special functions in the literature are observed and presented.

\section{Preliminaries}
The gamma function \cite{andrews} for $\Re(x)>0$ is defined by
\begin{align*}
	\Gamma(x)=\int_{0}^{\infty}t^{x-1}\exp(-t)dt.
\end{align*}
The beta function \cite{andrews} for $\Re(x)>0$, $\Re(y)>0$ is defined by
\begin{align*}
	B(x,y)=\int_{0}^{1} t^{x-1}(1-t)^{y-1}dt.
\end{align*}
The Pochhammer symbol \cite{andrews} for $\Re(\lambda)>-n$, $n\in\mathbb{N}$, $\lambda\neq0,-1,-2,\ldots$ is defined by
\begin{align*}
	(\lambda)_{n}=\frac{\Gamma(\lambda+n)}{\Gamma(\lambda)}, \quad(\lambda)_{0} \equiv 1.
\end{align*} 
The Gauss hypergeometric function \cite{kilbas} for $|z|<1$ is defined by
\begin{align*}
	{}_{2}F_{1}(\lambda_1,\lambda_2;\lambda_3;z)=\sum_{n=0}^{\infty} \frac{(\lambda_1)_n(\lambda_2)_{n}}{(\lambda_3)_{n}}\frac{z^n}{n!}.  
\end{align*}
The confluent hypergeometric function \cite{kilbas} is defined by
\begin{align*}
	\Phi(\lambda_{2};\lambda_{3};z)=\sum_{n=0}^{\infty} \frac{(\lambda_2)_{n}}{(\lambda_3)_{n}}\frac{z^n}{n!}.   
\end{align*}
The generalized M-series \cite{sharma1} for $\Re(\alpha)>0$, $\kappa_{1},\ldots,\kappa_{p}, \mu_{1},\ldots,\mu_{q}\neq 0,-1,-2,\ldots$ is defined by
\begin{align*}
	{}_{p}^{\alpha}\!M^{\beta}_{q}(\kappa_{1},\ldots,\kappa_{p};\mu_{1},\ldots,\mu_{q};z)=\sum_{m=0}^{\infty}\frac{(\kappa_{1})_{m}\ldots(\kappa_p)_{m}}{(\mu_1)_{m}\ldots(\mu_q)_{m}}\frac{z^{m}}{\Gamma(\alpha m+\beta)}.
\end{align*}
The Mellin and inverse Mellin transforms \cite{debnath} are defined by
\begin{align*}
	\mathfrak{M}\left\lbrace f(\rho);s \right\rbrace &=\hat{f}(s)=\int_{0}^{\infty}\rho^{s-1}f(\rho)d\rho,
	\\
	\mathfrak{M}^{-1}\left\lbrace \hat{f}(s) \right\rbrace &=f(\rho)=\frac{1}{2\pi i}\int_{c-i\infty}^{c+i\infty}\rho^{-s}\hat{f}(s)ds,\quad\left(c>0 \right) .
\end{align*}
The Laplace and inverse Laplace transforms \cite{debnath} are defined by
\begin{align*}
	\mathfrak{L}\left\lbrace f(\rho);s \right\rbrace &=F(s)=\int_{0}^{\infty}\exp(-s\rho)f(\rho)d\rho,
	\\
	\mathfrak{L}^{-1}\left\lbrace F(s) \right\rbrace &=f(\rho)=\frac{1}{2\pi i}\int_{c-i\infty}^{c+i\infty}\exp(s\rho)F(s)ds,\quad\left(c>0 \right) .
\end{align*} 
The beta transform \cite{beta} is defined by
\begin{align*}
	\mathfrak{B}\left\lbrace f(\rho);\omega,w \right\rbrace =\int_{0}^{1}\rho^{\omega-1}(1-\rho)^{w-1}f(\rho)d\rho.
\end{align*} 

Throughout the paper we assume that
 $\kappa_{1},\ldots,\kappa_{p}, \mu_{1},\ldots,\mu_{q}\neq 0,-1,-2,\ldots$, $\Re(x)>0$, $\Re(y)>0$, $\Re(\omega)>0$, $\Re(w)>0$, $\Re(\alpha)>0$, $\Re(\rho)>0$, $\Re(\lambda_3)>\Re(\lambda_2)>0$, $\Re(s)>0$, $\Re(x+s)>0$, $\Re(y+s)>0$.

\section{Modified gamma and beta functions and their properties}
Modified gamma and beta functions are given in this section. Later, some properties of these functions are presented.
\begin{definition}
	The modified gamma function is given as:
	\begin{align}\label{gm1}
		{}^{M}\Gamma_{p,q}^{(\alpha,\beta)}(x;\rho)&={}^{M}\Gamma_{p,q}^{(\alpha,\beta)}(\kappa_{1},\ldots,\kappa_p;\mu_1,\ldots,\mu_q;x;\rho)\nonumber
		\\
		&:=\int_{0}^{\infty}t^{x-1}~{}_{p}^{\alpha}\!M^{\beta}_{q}\left(\kappa_{1},\ldots,\kappa_{p};\mu_{1},\ldots,\mu_{q};-t-\frac{\rho}{t}\right) dt.
	\end{align}
\end{definition}

\begin{definition}
	The modified beta function is given as:
	\begin{align}\label{bt1}
		{}^{M}\!B_{p,q}^{(\alpha,\beta)}(x,y;\rho)
		&={}^{M}B_{p,q}^{(\alpha,\beta)}(\kappa_{1},\ldots,\kappa_p;\mu_1,\ldots,\mu_q;x,y;\rho)\nonumber
		\\
		&:=\int_{0}^{1}t^{x-1}(1-t)^{y-1}~{}_{p}^{\alpha}\!M^{\beta}_{q}\left(\kappa_{1},\ldots,\kappa_{p};\mu_{1},\ldots,\mu_{q};\frac{-\rho}{t(1-t)}\right) dt.
	\end{align}
\end{definition}
For the sake shortness, we call modified gamma and beta functions as M-gamma and M-beta functions, respectively.

\begin{theorem}
The Mellin transform of the M-beta function is obtained as:
	\begin{align*}
		\mathfrak{M}\left\lbrace  {}^{M}\!B_{p,q}^{(\alpha,\beta)}(x,y;\rho);s\right\rbrace =B(x+s,y+s)~{}^{M}\Gamma_{p,q}^{(\alpha,\beta)}(s).
	\end{align*}
\end{theorem}

\begin{proof}
	Using Mellin transform, we have
	\begin{align*}
		\mathfrak{M}\left\lbrace  {}^{M}\!B_{p,q}^{(\alpha,\beta)}(x,y;\rho);s\right\rbrace=\int_{0}^{1}t^{x-1}(1-t)^{y-1}\int_{0}^{\infty}\rho^{s-1}~{}_{p}^{\alpha}\!M^{\beta}_{q}\left(\kappa_{1},\ldots,\kappa_{p};\mu_{1},\ldots,\mu_{q};\frac{-\rho}{t(1-t)}\right)d\rho dt.
	\end{align*}
	In the second integral, by substituting $k=\frac{\rho}{t(1-t)}$, we have
	\begin{align*}
		\mathfrak{M}\left\lbrace  {}^{M}\!B_{p,q}^{(\alpha,\beta)}(x,y;\rho);s\right\rbrace &=\int_{0}^{1}t^{x+s-1}(1-t)^{y+s-1}dt\int_{0}^{\infty}k^{s-1}~{}_{p}^{\alpha}\!M^{\beta}_{q}\left(\kappa_{1},\ldots,\kappa_{p};\mu_{1},\ldots,\mu_{q};-k\right)dk
		\\
		&=B(x+s,y+s)~{}^{M}\Gamma_{p,q}^{(\alpha,\beta)}(s;0).\qedhere
	\end{align*} 
\end{proof}

\begin{corollary}
The relationship of the inverse Mellin transform with the M-beta function is obtained as:
	\begin{align*}
		{}^{M}\!B_{p,q}^{(\alpha,\beta)}(x,y;\rho)=\frac{1}{2\pi i}\int_{c-i\infty}^{c+i\infty}B(x+s,y+s)~{}^{M}\Gamma_{p,q}^{(\alpha,\beta)}(s;0)\rho^{-s}ds,\quad\left(c>0 \right) .
	\end{align*}
\end{corollary}

\begin{theorem}
The Laplace transform of the M-beta function is obtained as:
	\begin{align*}
		\mathfrak{L}\left\lbrace  {}^{M}\!B_{p,q}^{(\alpha,\beta)}(x,y;\rho);s\right\rbrace =\frac{1}{s}~{}^{M}\!B_{p+1,q}^{(\alpha,\beta)}\left( \kappa_{1},\ldots,\kappa_p,1;\mu_1,\ldots,\mu_q;x,y;\frac{1}{s}\right).
	\end{align*}
\end{theorem}

\begin{proof}
	Using Laplace transform, we have
	\begin{align*}
		\mathfrak{L}\left\lbrace  {}^{M}\!B_{p,q}^{(\alpha,\beta)}(x,y;\rho);s\right\rbrace &=\frac{1}{s}\int_{0}^{1}t^{x-1}(1-t)^{y-1}~{}_{p+1}^{\alpha}M^{\beta}_{q}\left(\kappa_{1},\ldots,\kappa_p,1;\mu_1,\ldots,\mu_q;\frac{-\frac{1}{s}}{t(1-t)} \right) dt
		\\
		&=\frac{1}{s}~{}^{M}\!B_{p+1,q}^{(\alpha,\beta)}\left( \kappa_{1},\ldots,\kappa_p,1;\mu_1,\ldots,\mu_q;x,y;\frac{1}{s}\right).\qedhere
	\end{align*}
\end{proof}

\begin{corollary}
The relationship of the inverse Laplace transform with the M-beta function is obtained as:
	\begin{align*}
		{}^{M}\!B_{p,q}^{(\alpha,\beta)}(x,y;\rho)=\frac{1}{2\pi i}\int_{c-i\infty}^{c+i\infty}\exp(s\rho)\frac{1}{s}~{}^{M}\!B_{p+1,q}^{(\alpha,\beta)}\left( \kappa_{1},\ldots,\kappa_p,1;\mu_1,\ldots,\mu_q;x,y;\frac{1}{s}\right)ds,\quad\left(c>0 \right) .
	\end{align*}
\end{corollary}

\begin{theorem}
The beta transform of the M-beta function is obtained as:
	\begin{align*}
		\mathfrak{B}\left\lbrace  {}^{M}\!B_{p,q}^{(\alpha,\beta)}(x,y;\rho);\omega,w\right\rbrace =B(\omega,w)~{}^{M}\!B_{p+1,q+1}^{(\alpha,\beta)}\left(\kappa_{1},\ldots,\kappa_p,\omega;\mu_1,\ldots,\mu_q,\omega+w; x,y;1\right).
	\end{align*}
\end{theorem}

\begin{proof}
	Using beta transform, we have
	\begin{align*}
		\mathfrak{B}&\left\lbrace {}^{M}\!B_{p,q}^{(\alpha,\beta)}(x,y;\rho);\omega,w\right\rbrace 
		\\
		&=B(\omega,w)\int_{0}^{1}t^{x-1}(1-t)^{y-1}~{}_{p+1}^{\alpha}M^{\beta}_{q+1}\left(\kappa_{1},\ldots,\kappa_p,\omega;\mu_1,\ldots,\mu_q,\omega+w;\frac{-1}{t(1-t)} \right) dt
		\\
		&=B(\omega,w)~{}^{M}\!B_{p+1,q+1}^{(\alpha,\beta)}\left(\kappa_{1},\ldots,\kappa_p,\omega;\mu_1,\ldots,\mu_q,\omega+w; x,y;1\right).\qedhere
	\end{align*}
\end{proof}

\begin{theorem}
The integral representations of the M-beta function is obtained as:
	\begin{align*}
		{}^{M}\!B_{p,q}^{(\alpha,\beta)}(x,y;\rho)&=2\int_{0}^{\frac{\pi}{2}} \cos^{2x-1}(\theta) \sin^{2y-1}(\theta)~{}_{p}^{\alpha}\!M^{\beta}_{q}\left(\kappa_{1},\ldots,\kappa_{p};\mu_{1},\ldots,\mu_{q};\frac{-\rho}{\sin^{2}(\theta)\cos^{2}(\theta)}\right)d\theta,
		\\
		{}^{M}\!B_{p,q}^{(\alpha,\beta)}(x,y;\rho)&=\int_{0}^{\infty}\frac{u^{x-1}}{(1+u)^{x+y}}~{}_{p}^{\alpha}\!M^{\beta}_{q}\left(\kappa_{1},\ldots,\kappa_{p};\mu_{1},\ldots,\mu_{q};-2\rho-\rho\left( u+\frac{1}{u}\right)\right)du.
	\end{align*}
\end{theorem} 
\begin{proof}
	Taking $t= \cos^{2}(\theta)$ in \eqref{bt1}, we get
	\begin{align*}
		{}^{M}\!B_{p,q}^{(\alpha,\beta)}(x,y;\rho)&=\int_{0}^{1}t^{x-1}(1-t)^{y-1}~{}_{p}^{\alpha}\!M^{\beta}_{q}\left(\kappa_{1},\ldots,\kappa_{p};\mu_{1},\ldots,\mu_{q};\frac{-\rho}{t(1-t)}\right) dt\\
		&=2\int_{0}^{\frac{\pi}{2}} \cos^{2x-1}(\theta) \sin^{2y-1}(\theta)~{}_{p}^{\alpha}\!M^{\beta}_{q}\left(\kappa_{1},\ldots,\kappa_{p};\mu_{1},\ldots,\mu_{q};\frac{-\rho}{\sin^{2}(\theta)\cos^{2}(\theta)}\right)d\theta.
	\end{align*}
	Taking $t=\frac{u}{1+u}$ in \eqref{bt1}, we get 
	\begin{align*}
		{}^{M}\!B_{p,q}^{(\alpha,\beta)}(x,y;\rho)&=\int_{0}^{1}t^{x-1}(1-t)^{y-1}~{}_{p}^{\alpha}\!M^{\beta}_{q}\left(\kappa_{1},\ldots,\kappa_{p};\mu_{1},\ldots,\mu_{q};\frac{-\rho}{t(1-t)}\right) dt\\
		&=\int_{0}^{\infty} \frac{u^{x-1}}{(1+u)^{x+y}}~{}_{p}^{\alpha}\!M^{\beta}_{q}\left(\kappa_{1},\ldots,\kappa_{p};\mu_{1},\ldots,\mu_{q};-2\rho-\rho\left( u+\frac{1}{u}\right)\right)du.\qedhere
	\end{align*}
\end{proof}

\begin{theorem}
The relationship of the M-beta functions is obtained as:
	\begin{align*}
		{}^{M}\!B_{p,q}^{(\alpha,\beta)}(x,y+1;\rho)+{}^{M}\!B_{p,q}^{(\alpha,\beta)}(x+1,y;\rho)={}^{M}\!B_{p,q}^{(\alpha,\beta)}(x,y;\rho).
	\end{align*}
\end{theorem} 
\begin{proof}
Using equation \eqref{bt1}, we have	
	\begin{align*}
		{}^{M}\!B_{p,q}^{(\alpha,\beta)}&(x,y+1;\rho)+{}^{M}\!B_{p,q}^{(\alpha,\beta)}(x+1,y;\rho)\\
		&=\int_{0}^{1}t^{x-1}(1-t)^{y}~{}_{p}^{\alpha}\!M^{\beta}_{q}\left(\kappa_{1},\ldots,\kappa_{p};\mu_{1},\ldots,\mu_{q};\frac{-\rho}{t(1-t)}\right) dt
		\\
		&~~~~+\int_{0}^{1}t^{x}(1-t)^{y-1}~{}_{p}^{\alpha}\!M^{\beta}_{q}\left(\kappa_{1},\ldots,\kappa_{p};\mu_{1},\ldots,\mu_{q};\frac{-\rho}{t(1-t)}\right) dt\\
		&=\int_{0}^{1}\big( t^{x-1}(1-t)^{y}+t^{x}(1-t)^{y-1}\big){}_{p}^{\alpha}\!M^{\beta}_{q}\left(\kappa_{1},\ldots,\kappa_{p};\mu_{1},\ldots,\mu_{q};\frac{-\rho}{t(1-t)}\right) dt\\
		&=\int_{0}^{1}t^{x-1}(1-t)^{y-1}~{}_{p}^{\alpha}\!M^{\beta}_{q}\left(\kappa_{1},\ldots,\kappa_{p};\mu_{1},\ldots,\mu_{q};\frac{-\rho}{t(1-t)}\right) dt\\
		&={}^{M}\!B_{p,q}^{(\alpha,\beta)}(x,y;\rho).\qedhere
	\end{align*}	
\end{proof}

\begin{theorem}
The relationship of the M-gamma functions is obtained as:
	\begin{align*}
		{}^{M}\Gamma_{p,q}^{(\alpha,\beta)}(x;\rho){}^{M}\Gamma_{p,q}^{(\alpha,\beta)}(y;\rho)&=4\int_{0}^{\frac{\pi}{2}}\int_{0}^{\infty}r^{2(x+y)-1} \cos^{2x-1}(\theta) \sin^{2y-1}(\theta)\\
		&~~~\times{}_{p}^{\alpha}\!M^{\beta}_{q}\left(\kappa_{1},\ldots,\kappa_{p};\mu_{1},\ldots,\mu_{q};-r^{2} \cos^{2}(\theta)-\frac{\rho}{r^{2} \cos^{2}(\theta)}\right)
		\\
		&~~~\times{}_{p}^{\alpha}\!M^{\beta}_{q}\left(\kappa_{1},\ldots,\kappa_{p};\mu_{1},\ldots,\mu_{q};-r^{2} \sin^{2}(\theta)-\frac{\rho}{r^{2} \sin^{2}(\theta)}\right)drd\theta.
	\end{align*}
\end{theorem} 
\begin{proof}
	Taking $t=\eta^{2}$ in \eqref{gm1}, we get
	\begin{align*}
		{}^{M}\Gamma_{p,q}^{(\alpha,\beta)}(x;\rho)=2\int_{0}^{\infty}\eta^{2x-1}~{}_{p}^{\alpha}\!M^{\beta}_{q}\left(\kappa_{1},\ldots,\kappa_{p};\mu_{1},\ldots,\mu_{q};-\eta^{2}-\frac{\rho}{\eta^{2}}\right)d \eta.
	\end{align*}
	Therefore,
	\begin{align*}
		{}^{M}\Gamma\kappa_{p}^{(\alpha,\beta)}(x;\rho){}^{M}\Gamma\kappa_{p}^{(\alpha,\beta)}(y;\rho)&=4\int_{0}^{\infty}\int_{0}^{\infty}\eta^{2x-1}\xi^{2y-1}~{}_{p}^{\alpha}\!M^{\beta}_{q}\left(\kappa_{1},\ldots,\kappa_{p};\mu_{1},\ldots,\mu_{q};-\eta^{2}-\frac{\rho}{\eta^{2}}\right)
		\\
		&~~~\times{}_{p}^{\alpha}\!M^{\beta}_{q}\left(\kappa_{1},\ldots,\kappa_{p};\mu_{1},\ldots,\mu_{q};-\xi^{2}-\frac{\rho}{\xi^{2}}\right)d\eta d\xi.
	\end{align*}
	Taking $\eta=r \cos(\theta)$ and $\xi=r \sin(\theta)$ in above equality,
	\begin{align*}
		{}^{M}\Gamma_{p,q}^{(\alpha,\beta)}(x;\rho){}^{M}\Gamma_{p,q}^{(\alpha,\beta)}(y;\rho)&=4\int_{0}^{\frac{\pi}{2}}\int_{0}^{\infty}r^{2(x+y)-1} \cos^{2x-1}(\theta) \sin^{2y-1}(\theta)\\
		&~~~\times{}_{p}^{\alpha}\!M^{\beta}_{q}\left(\kappa_{1},\ldots,\kappa_{p};\mu_{1},\ldots,\mu_{q};-r^{2} \cos^{2}(\theta)-\frac{\rho}{r^{2} \cos^{2}(\theta)}\right)
		\\
		&~~~\times{}_{p}^{\alpha}\!M^{\beta}_{q}\left(\kappa_{1},\ldots,\kappa_{p};\mu_{1},\ldots,\mu_{q};-r^{2} \sin^{2}(\theta)-\frac{\rho}{r^{2} \sin^{2}(\theta)}\right)drd\theta.\qedhere
	\end{align*}
\end{proof}

\begin{theorem}
The M-beta function has the equation:
	\begin{align*}
		{}^{M}\!B_{p,q}^{(\alpha,\beta)}(x,1-y;\rho)=\sum_{n=0}^{\infty}\frac{(y)_{n}}{n!}~{}^{M}\!B_{p}^{(\alpha,\beta)}(x+n,1;\rho).
	\end{align*}
\end{theorem} 
\begin{proof}
	From equation \eqref{bt1}, we have
	\begin{align*}
		{}^{M}\!B_{p,q}^{(\alpha,\beta)}(x,1-y;\rho)=\int_{0}^{1}t^{x-1}(1-t)^{-y}~{}_{p}^{\alpha}\!M^{\beta}_{q}\left(\kappa_{1},\ldots,\kappa_{p};\mu_{1},\ldots,\mu_{q};\frac{-\rho}{t(1-t)}\right) dt.
	\end{align*}
	Binomial series \cite{andrews} is defined by
	\begin{align*}
		(1-t)^{-y}=\sum_{n=0}^{\infty}(y)_{n}\frac{t^{n}}{n!},\quad \big(|t|<1\big).
	\end{align*}
	Considering binomial series, we have
	\begin{align*}
		{}^{M}\!B_{p,q}^{(\alpha,\beta)}(x,1-y;\rho)&=\sum_{n=0}^{\infty}\frac{(y)_{n}}{n!}~ {}^{M}\!B_{p,q}^{(\alpha,\beta)}(x+n,1;\rho).\qedhere
	\end{align*}
\end{proof}

\section{Modified Gauss and confluent hypergeometric functions and their properties}
	Modified Gauss hypergeometric and confluent hypergeometric functions are given in this section. Later, some properties of these functions are presented.
\begin{definition}
	Modified Gauss hypergeometric function is given as:
	\begin{align*}
		{}^{M}\!F_{p,q}^{(\alpha,\beta)}(\lambda_1,\lambda_2;\lambda_3;z;\rho)&={}^{M}F_{p,q}^{(\alpha,\beta)}(\kappa_{1},\ldots,\kappa_p;\mu_1,\ldots,\mu_q;\lambda_1,\lambda_2;\lambda_3;z;\rho)
		\\
		&:=\sum_{n=0}^{\infty}(\lambda_1)_{n}\frac{{}^{M}\!B_{p,q}^{(\alpha,\beta)}(\lambda_2+n,\lambda_3-\lambda_2;\rho)}{B(\lambda_2,\lambda_3-\lambda_2)}\frac{z^{n}}{n!}.
	\end{align*}
\end{definition}

\begin{definition}
	Modified confluent hypergeometric function is given as:
	\begin{align*}
		{}^{M}\!\Phi_{p,q}^{(\alpha,\beta)}(\lambda_{2};\lambda_{3};z;\rho)&={}^{M}\Phi_{p,q}^{(\alpha,\beta)}(\kappa_{1},\ldots,\kappa_p;\mu_1,\ldots,\mu_q;\lambda_{2};\lambda_{3};z;\rho)
		\\
		&:=\sum_{n=0}^{\infty}\frac{{}^{M}\!B_{p,q}^{(\alpha,\beta)}(\lambda_{2}+n,\lambda_{3}-\lambda_{2};\rho)}{B(\lambda_{2},\lambda_{3}-\lambda_{2})}\frac{z^{n}}{n!}.
	\end{align*}
\end{definition}
For the sake shortness, we call modified Gauss hypergeometric and confluent hypergeometric functions as M-Gauss hypergeometric and M-confluent hypergeometric functions, respectively.

\begin{theorem}
The Mellin transform of the M-Gauss hypergeometric function is obtained as:
	\begin{align*}
	\mathfrak{M}\left\lbrace  {}^{M}\!F_{p,q}^{(\alpha,\beta)}(\lambda_{1},\lambda_{2};\lambda_{3};z;\rho);s\right\rbrace =\frac{{}^{M}\Gamma_{p,q}^{(\alpha,\beta)}(s)B(\lambda_{2}+s,\lambda_{3}+s-\lambda_{2})}{B(\lambda_{2},\lambda_{3}-\lambda_{2})}{}_{2}F_{1}(\lambda_{1},\lambda_{2}+s;\lambda_{3}+2s;z).
	\end{align*}
\end{theorem}
\begin{proof}
	Using Mellin transform, we have
	\begin{align*}
	&\mathfrak{M}\big\lbrace {}^{M}\!F_{p,q}^{(\alpha,\beta)}(\lambda_{1},\lambda_{2};\lambda_{3};z;\rho);s\big\rbrace
	\\
	&=\int_{0}^{\infty}\rho^{s-1} ~{}^{M}\!F_{p,q}^{(\alpha,\beta)}(\lambda_{1},\lambda_{2};\lambda_{3};z;\rho)d\rho
	\\
	&=\frac{1}{B(\lambda_{2},\lambda_{3}\!-\!\lambda_{2})}\!\int_{0}^{1}\!\!t^{\lambda_{2}-1}(1\!-\!t)^{\lambda_{3}-\lambda_{2}-1}(1\!-\!zt)^{-\lambda_{1}}\!\!\int_{0}^{\infty}\!\!\rho^{s-1}~{}_{p}^{\alpha}\!M^{\beta}_{q}\left(\kappa_{1},\ldots,\kappa_{p};\mu_{1},\ldots,\mu_{q};\frac{-\rho}{t(1\!-\!t)}\right)d\rho dt.
	\end{align*}
	In second integral, by substituting $k=\frac{\rho}{t(1-t)}$, we have
	\begin{align*}
	\int_{0}^{\infty}\rho^{s-1}~{}_{p}^{\alpha}\!M^{\beta}_{q}\left(\kappa_{1},\ldots,\kappa_{p};\mu_{1},\ldots,\mu_{q};\frac{-\rho}{t(1-t)}\right)d\rho=t^{s}(1-t)^{s}~{}^{M}\Gamma_{p,q}^{(\alpha,\beta)}(s;0).
	\end{align*}
	Thus, we get 
	\begin{equation*}
	\mathfrak{M}\left\lbrace  {}^{M}\!F_{p,q}^{(\alpha,\beta)}(\lambda_{1},\lambda_{2};\lambda_{3};z;\rho);s\right\rbrace =\frac{{}^{M}\Gamma_{p,q}^{(\alpha,\beta)}(s;0)B(\lambda_{2}+s,\lambda_{3}+s-\lambda_{2})}{B(\lambda_{2},\lambda_{3}-\lambda_{2})}{}_{2}F_{1}(\lambda_{1},\lambda_{2}+s;\lambda_{3}+2s;z).\qedhere
	\end{equation*}
\end{proof}

\begin{corollary}
The relationship of the inverse Mellin transform with the M-Gauss hypergeometric function for $c>0$ is obtained as:
	\begin{align*}
	{}^{M}\!F_{p,q}^{(\alpha,\beta)}(\lambda_{1},\lambda_{2};\lambda_{3};z;\rho)=\frac{1}{2\pi i}\int_{c-i\infty}^{c+i\infty}\frac{{}^{M}\Gamma_{p,q}^{(\alpha,\beta)}(s;0)B(\lambda_{2}\!+\!s,\lambda_{3}\!+\!s\!-\!\lambda_{2})}{B(\lambda_{2},\lambda_{3}\!-\!\lambda_{2})}{}_{2}F_{1}(\lambda_{1},\lambda_{2}\!+\!s;\lambda_{3}\!+\!2s;z)\rho^{-s}ds.
	\end{align*}
\end{corollary}

\begin{theorem}
The Mellin transform of the M-confluent hypergeometric function is obtained as:
	\begin{align*}
	\mathfrak{M}\left\lbrace  {}^{M}\!\Phi_{p,q}^{(\alpha,\beta)}(\lambda_{2};\lambda_{3};z;\rho);s\right\rbrace =\frac{{}^{M}\Gamma_{p,q}^{(\alpha,\beta)}(s;0)B(\lambda_{2}+s,\lambda_{3}+s-\lambda_{2})}{B(\lambda_{2},\lambda_{3}-\lambda_{2})}\Phi(\lambda_{2}+s;\lambda_{3}+2s;z).
	\end{align*}
\end{theorem}
\begin{proof}
	By making similar calculations, the desired result is achieved.
\end{proof}

\begin{corollary}
The relationship of the inverse Mellin transform with the M-confluent hypergeometric function for $c>0$ is obtained as:
	\begin{align*}
	{}^{M}\!\Phi_{p,q}^{(\alpha,\beta)}(\lambda_{2};\lambda_{3};z;\rho)=\frac{1}{2\pi i}\int_{c-i\infty}^{c+i\infty}\frac{{}^{M}\Gamma_{p,q}^{(\alpha,\beta)}(s;0)B(\lambda_{2}+s,\lambda_{3}+s-\lambda_{2})}{B(\lambda_{2},\lambda_{3}-\lambda_{2})}\Phi(\lambda_{2}+s;\lambda_{3}+2s;z)p^{-s}ds.
	\end{align*}
\end{corollary}

\begin{theorem}
The Laplace transform of the M-Gauss hypergeometric function is obtained as:
	\begin{align*}
	\mathfrak{L}\left\lbrace  {}^{M}\!F_{p,q}^{(\alpha,\beta)}(\lambda_{1},\lambda_{2};\lambda_{3};z;\rho);s\right\rbrace =\frac{1}{s}~{}^{M}\!F_{p+1,q}^{(\alpha,\beta)}\left( \kappa_{1},\ldots,\kappa_p,1;\mu_1,\ldots,\mu_q;\lambda_{1},\lambda_{2};\lambda_{3};z;\frac{1}{s}\right).
	\end{align*}
\end{theorem}
\begin{proof}
	Using Laplace transform, we have
	\begin{align*}
	&\mathfrak{L}\big\lbrace {}^{M}\!F_{p,q}^{(\alpha,\beta)}(\lambda_{1},\lambda_{2};\lambda_{3};z;\rho);s\big\rbrace
	\\
	&=\int_{0}^{\infty}\exp(-s\rho) ~{}^{M}\!F_{p,q}^{(\alpha,\beta)}(\lambda_{1},\lambda_{2};\lambda_{3};z;\rho)d\rho
	\\
	&=\frac{1}{s}\frac{1}{B(\lambda_{2},\lambda_{3}-\lambda_{2})}\int_{0}^{1}t^{\lambda_{2}-1}(1-t)^{\lambda_{3}-\lambda_{2}-1}(1-zt)^{-\lambda_{1}}~{}_{p+1}^{\alpha}\!M^{\beta}_{q}\left(\kappa_{1},\ldots,\kappa_{p},1;\mu_{1},\ldots,\mu_{q};\frac{-\frac{1}{s}}{t(1-t)}\right)dt
	\\
	&=\frac{1}{s}~{}^{M}\!F_{p+1,q}^{(\alpha,\beta)}\left( \kappa_{1},\ldots,\kappa_p,1;\mu_1,\ldots,\mu_q;\lambda_{1},\lambda_{2};\lambda_{3};z;\frac{1}{s}\right).\qedhere
	\end{align*}
\end{proof}

\begin{corollary}
The relationship of the inverse Laplace transform with the M-Gauss hypergeometric function for $c>0$ is obtained as:
	\begin{align*}
	{}^{M}\!F_{p,q}^{(\alpha,\beta)}(\lambda_{1},\lambda_{2};\lambda_{3};z;\rho)=\frac{1}{2\pi i}\int_{c-i\infty}^{c+i\infty}\exp(s\rho)\frac{1}{s}~{}^{M}\!F_{p+1,q}^{(\alpha,\beta)}\left( \kappa_{1},\ldots,\kappa_p,1;\mu_1,\ldots,\mu_q;\lambda_{1},\lambda_{2};\lambda_{3};z;\frac{1}{s}\right)ds.
	\end{align*}
\end{corollary}

\begin{theorem}
The Laplace transform of the M-confluent hypergeometric function is obtained as:
	\begin{align*}
	\mathfrak{L}\left\lbrace  {}^{M}\!\Phi_{p,q}^{(\alpha,\beta)}(\lambda_{2};\lambda_{3};z;\rho);s\right\rbrace =\frac{1}{s}~{}^{M}\!\Phi_{p+1,q}^{(\alpha,\beta)}\left(\kappa_{1},\ldots,\kappa_p,1;\mu_1,\ldots,\mu_q;\lambda_{2};\lambda_{3};z;\frac{1}{s}\right).
	\end{align*}
\end{theorem}
\begin{proof}
	By making similar calculations, the desired result is achieved.
\end{proof}

\begin{corollary}
The relationship of the inverse Laplace transform with the M-Gauss hypergeometric function for $c>0$ is obtained as:
	\begin{align*}
	{}^{M}\!\Phi_{p,q}^{(\alpha,\beta)}(\lambda_{2};\lambda_{3};z;\rho)=\frac{1}{2\pi i}\int_{c-i\infty}^{c+i\infty}\exp(s\rho)\frac{1}{s}~{}^{M}\!\Phi_{p+1,q}^{(\alpha,\beta)}\left(\kappa_{1},\ldots,\kappa_p,1;\mu_1,\ldots,\mu_q;\lambda_{2};\lambda_{3};z;\frac{1}{s}\right)ds.
	\end{align*}
\end{corollary}

\begin{theorem}
The beta transform of the M-Gauss hypergeometric function is obtained as:
	\begin{align*}
	&\mathfrak{B}\left\lbrace  {}^{M}\!F_{p,q}^{(\alpha,\beta)}(\lambda_{1},\lambda_{2};\lambda_{3};z;\rho);\omega,w\right\rbrace 
	\\
	&\quad\quad\quad\quad\quad\quad=B(\omega,w)~{}^{M}\!F_{p+1,q+1}^{(\alpha,\beta)}\left(\kappa_{1},\ldots,\kappa_p,\omega;\mu_1,\ldots,\mu_q,\omega+w; \lambda_{1},\lambda_{2};\lambda_{3};z;1\right).
	\end{align*}
\end{theorem}
\begin{proof}
	Using beta transform, we have
	\begin{align*}
	\mathfrak{B}\big\lbrace {}^{M}\!F_{p,q}^{(\alpha,\beta)}&(\lambda_{1},\lambda_{2};\lambda_{3};z;\rho);\omega,w\big\rbrace
	\\
	&=\int_{0}^{1}\rho^{\omega-1}(1-\rho)^{w-1}~{}^{M}\!F_{p,q}^{(\alpha,\beta)}(\lambda_{1},\lambda_{2};\lambda_{3};z;\rho)d\rho
	\\
	&=\frac{B(\omega,w)}{B(\lambda_{2},\lambda_{3}-\lambda_{2})}\int_{0}^{1}t^{\lambda_{2}-1}(1-t)^{\lambda_{3}-\lambda_{2}-1}(1-zt)^{-\lambda_{1}}
	\\
	&\quad\times{}_{p+1}^{\alpha}\!M^{\beta}_{q+1}\left(\kappa_{1},\ldots,\kappa_{p},\omega;\mu_{1},\ldots,\mu_{q},\omega+w;\frac{-1}{t(1-t)}\right)dt
	\\
	&=B(\omega,w)~{}^{M}\!F_{p+1,q+1}^{(\alpha,\beta)}\left(\kappa_{1},\ldots,\kappa_p,\omega;\mu_1,\ldots,\mu_q,\omega+w; \lambda_{1},\lambda_{2};\lambda_{3};z;1\right).\qedhere
	\end{align*}
\end{proof}

\begin{theorem}
The beta transform of the M-confluent hypergeometric function is obtained as:
	\begin{align*}
	\mathfrak{B}\left\lbrace  {}^{M}\!\Phi_{p,q}^{(\alpha,\beta)}(\lambda_{2};\lambda_{3};z;\rho);\omega,w\right\rbrace =B(\omega,w)~{}^{M}\!\Phi_{p+1,q+1}^{(\alpha,\beta)}\left(\kappa_{1},\ldots,\kappa_p,\omega;\mu_1,\ldots,\mu_q,\omega+w;\lambda_{2};\lambda_{3};z;1\right).
	\end{align*}
\end{theorem}
\begin{proof}
	By making similar calculations, the desired result is achieved.
\end{proof}

\begin{theorem} The integral representations of the M-Gauss hypergeometric function is obtained as: 
	\begin{align}\label{j1}
		{}^{M}\!F_{p,q}^{(\alpha,\beta)}(\lambda_{1},\lambda_{2};\lambda_{3};z;\rho)&=\frac{1}{B(\lambda_{2},\lambda_{3}-\lambda_{2})}\int_{0}^{1}t^{\lambda_{2}-1}(1-t)^{\lambda_{3}-\lambda_{2}-1}(1-zt)^{-\lambda_{1}}\nonumber
		\\
		&~~~~\times{}_{p}^{\alpha}\!M^{\beta}_{q}\left(\kappa_{1},\ldots,\kappa_{p};\mu_{1},\ldots,\mu_{q};\frac{-\rho}{t(1-t)}\right)dt,
		\\
		{}^{M}\!F_{p,q}^{(\alpha,\beta)}(\lambda_{1},\lambda_{2};\lambda_{3};z;\rho)&=\frac{1}{B(\lambda_{2},\lambda_{3}-\lambda_{2})}\int_{0}^{\infty}u^{\lambda_{2}-1}(1+u)^{\lambda_{1}-\lambda_{3}}\big( 1+u(1-z)\big)^{-\lambda_{1}}\nonumber
		\\
		&~~~\times{}_{p}^{\alpha}\!M^{\beta}_{q}\left(\kappa_{1},\ldots,\kappa_{p};\mu_{1},\ldots,\mu_{q};-2\rho-\rho\left( u+\frac{1}{u}\right)\right)du,\nonumber
		\\
		{}^{M}\!F_{p,q}^{(\alpha,\beta)}(\lambda_{1},\lambda_{2};\lambda_{3};z;\rho)&=\frac{2}{B(\lambda_{2},\lambda_{3}-\lambda_{2})}\int_{0}^{\frac{\pi}{2}} \sin^{2\lambda_{2}-1}(\theta) \cos^{2\lambda_{3}-2\lambda_{2}-1}(\theta)\left( 1-z \sin^{2}(\theta)\right) ^{-\lambda_{1}}\nonumber
		\\
		&~~~\times{}_{p}^{\alpha}\!M^{\beta}_{q}\left(\kappa_{1},\ldots,\kappa_{p};\mu_{1},\ldots,\mu_{q};\frac{-\rho}{ \sin^{2}(\theta) \cos^{2}(\theta)}\right)d\theta.\nonumber
	\end{align}
\end{theorem}

\begin{proof}
	Considering binomial series, we have	
	\begin{align*}
		{}^{M}\!&F_{p,q}^{(\alpha,\beta)}(\lambda_{1},\lambda_{2};\lambda_{3};z;\rho)
		\\
		&=\sum_{n=0}^{\infty}(\lambda_1)_{n}\frac{{}^{M}\!B_{p,q}^{(\alpha,\beta)}(\lambda_{2}+n,\lambda_{3}-\lambda_{2};\rho)}{B(\lambda_{2},\lambda_{3}-\lambda_{2})}\frac{z^{n}}{n!}
		\\
		&=\frac{1}{B(\lambda_{2},\lambda_{3}-\lambda_{2})}\int_{0}^{1}t^{\lambda_{2}-1}(1-t)^{\lambda_{3}-\lambda_{2}-1}(1-zt)^{-\lambda_{1}}~{}_{p}^{\alpha}\!M^{\beta}_{q}\left(\kappa_{1},\ldots,\kappa_{p};\mu_{1},\ldots,\mu_{q};\frac{-\rho}{t(1-t)}\right)dt.
	\end{align*}
	Taking $u=\frac{t}{1-t}$ in \eqref{j1}, we get
	\begin{align*}
		{}^{M}\!F_{p,q}^{(\alpha,\beta)}(\lambda_{1},\lambda_{2};\lambda_{3};z;\rho)&=\frac{1}{B(\lambda_{2},\lambda_{3}-\lambda_{2})}\int_{0}^{\infty}u^{\lambda_{2}-1}(1+u)^{\lambda_{1}-\lambda_{3}}\big( 1+u(1-z)\big)^{-\lambda_{1}}
		\\
		&~~~~\times{}_{p}^{\alpha}\!M^{\beta}_{q}\left(\kappa_{1},\ldots,\kappa_{p};\mu_{1},\ldots,\mu_{q};-2\rho-\rho\left( u+\frac{1}{u}\right)\right)du.
	\end{align*}
	Taking $t= \sin^{2}(\theta)$ in \eqref{j1}, we have
	\begin{align*}
		{}^{M}\!F_{p,q}^{(\alpha,\beta)}(\lambda_{1},\lambda_{2};\lambda_{3};z;\rho)&=\frac{2}{B(\lambda_{2},\lambda_{3}-\lambda_{2})}\int_{0}^{\frac{\pi}{2}} \sin^{2\lambda_{2}-1}(\theta) \cos^{2\lambda_{3}-2\lambda_{2}-1}(\theta)\left( 1-z \sin^{2}(\theta)\right) ^{-\lambda_{1}}
		\\
		&~~~\times{}_{p}^{\alpha}\!M^{\beta}_{q}\left(\kappa_{1},\ldots,\kappa_{p};\mu_{1},\ldots,\mu_{q};\frac{-\rho}{ \sin^{2}(\theta) \cos^{2}(\theta)}\right)d\theta.\qedhere
	\end{align*} 
\end{proof}

\begin{theorem}
The integral representations of the M-confluent hypergeometric function is obtained as: 	
	\begin{align*}
		{}^{M}\!\Phi_{p,q}^{(\alpha,\beta)}(\lambda_{2};\lambda_{3};z;\rho)&=\frac{1}{B(\lambda_{2},\lambda_{3}-\lambda_{2})}\int_{0}^{1}t^{\lambda_{2}-1}(1-t)^{\lambda_{3}-\lambda_{2}-1}\exp(zt)
		\\
		&\quad\times{}_{p}^{\alpha}\!M^{\beta}_{q}\left(\kappa_{1},\ldots,\kappa_{p};\mu_{1},\ldots,\mu_{q};\frac{-\rho}{t(1-t)}\right)dt,
		\\
		{}^{M}\!\Phi_{p,q}^{(\alpha,\beta)}(\lambda_{2};\lambda_{3};z;\rho)&=\frac{1}{B(\lambda_{2},\lambda_{3}-\lambda_{2})}\int_{0}^{1}u^{\lambda_{3}-\lambda_{2}-1}(1-u)^{\lambda_{2}-1}\exp\big(z(1-u)\big)
		\\
		&\quad\times{}_{p}^{\alpha}\!M^{\beta}_{q}\left(\kappa_{1},\ldots,\kappa_{p};\mu_{1},\ldots,\mu_{q};\frac{-\rho}{u(1-u)}\right)du.
	\end{align*}
\end{theorem} 

\begin{proof}
	With similar calculations, desired results are obtained.
\end{proof}
Beta function and Pochhammer symbol equations are provided, respectively:
\begin{align*}
	B(\lambda_{2},\lambda_{3}-\lambda_{2})&=\frac{\lambda_{3}}{\lambda_{2}}B(\lambda_{2}+1,\lambda_{3}-\lambda_{2}),
	\\
	(\lambda_{1})_{n+1}&=\lambda_{1}(\lambda_{1}+1)_{n}.
\end{align*}
These equations are used in proof of two theorems given below.
\begin{theorem}
The $n$-order derivative of the M-Gauss hypergeometric function is obtained as:
	\begin{align*}
		\frac{d^{n}}{dz^{n}}\left\lbrace {}^{M}\!F_{p,q}^{(\alpha,\beta)}(\lambda_{1},\lambda_{2};\lambda_{3};z;\rho)\right\rbrace =\frac{(\lambda_1)_{n}(\lambda_2)_{n}}{(\lambda_3)_{n}}\left( {}^{M}\!F_{p,q}^{(\alpha,\beta)}(\lambda_{1}+n,\lambda_{2}+n;\lambda_{3}+n;z;\rho)\right).
	\end{align*}
\end{theorem}
\begin{proof}
	Differentiating M-Gauss hypergeometric function, we have
	\begin{align*}
		\frac{d}{dz}\left\lbrace {}^{M}\!F_{p,q}^{(\alpha,\beta)}(\lambda_{1},\lambda_{2};\lambda_{3};z;\rho)\right\rbrace &=\frac{d}{dz}\left\lbrace \sum_{n=0}^{\infty}(\lambda_1)_{n}\frac{{}^{M}\!B_{p,q}^{(\alpha,\beta)}(\lambda_{2}+n,\lambda_{3}-\lambda_{2};\rho)}{B(\lambda_{2},\lambda_{3}-\lambda_{2})}\frac{z^{n}}{n!}\right\rbrace \\
		&=\sum_{n=1}^{\infty}(\lambda_1)_{n}\frac{{}^{M}\!B_{p,q}^{(\alpha,\beta)}(\lambda_{2}+n,\lambda_{3}-\lambda_{2};\rho)}{B(\lambda_{2},\lambda_{3}-\lambda_{2})}\frac{z^{n-1}}{(n-1)!}. 
	\end{align*}
	By writing $n\to n+1$, we get
	\begin{align*}
		\frac{d}{dz}\left\lbrace {}^{M}\!F_{p,q}^{(\alpha,\beta)}(\lambda_{1},\lambda_{2};\lambda_{3};z;\rho)\right\rbrace &=\frac{(\lambda_{1})(\lambda_{2})}{(\lambda_{3})}\sum_{n=0}^{\infty}(\lambda_{1}+1)_{n}\frac{{}^{M}\!B_{p,q}^{(\alpha,\beta)}(\lambda_{2}+n+1,\lambda_{3}-\lambda_{2};\rho)}{B(\lambda_{2}+1,\lambda_{3}-\lambda_{2})}\frac{z^{n}}{n!}\\
		&=\frac{(\lambda_{1})(\lambda_{2})}{(\lambda_{3})}\left( {}^{M}\!F_{p,q}^{(\alpha,\beta)}(\lambda_{1}+1,\lambda_{2}+1;\lambda_{3}+1;z;\rho)\right).
	\end{align*}
	More general form:
	\begin{equation*}
		\frac{d^{n}}{dz^{n}}\left\lbrace {}^{M}\!F_{p,q}^{(\alpha,\beta)}(\lambda_{1},\lambda_{2};\lambda_{3};z;\rho)\right\rbrace =\frac{(\lambda_1)_{n}(\lambda_2)_{n}}{(\lambda_3)_{n}}\left({}^{M}\!F_{p,q}^{(\alpha,\beta)}(\lambda_{1}+n,\lambda_{2}+n;\lambda_{3}+n;z;\rho)\right).\qedhere
	\end{equation*}
\end{proof}

\begin{theorem}
The $n$-order derivative of the M-confluent hypergeometric function is obtained as:
	\begin{align*}
		\frac{d^{n}}{dz^{n}}\left\lbrace {}^{M}\!\Phi_{p,q}^{(\alpha,\beta)}(\lambda_{2};\lambda_{3};z;\rho)\right\rbrace =\frac{(\lambda_2)_{n}}{(\lambda_3)_{n}}\left( {}^{M}\!\Phi_{p,q}^{(\alpha,\beta)}(\lambda_{2}+n;\lambda_{3}+n;z;\rho)\right).
	\end{align*}
\end{theorem} 
\begin{proof}
	By making similar calculations, the desired result is achieved.
\end{proof}

\begin{theorem}
The M-Gauss hypergeometric function has the equation:
	\begin{align*}
		{}^{M}\!F_{p,q}^{(\alpha,\beta)}(\lambda_{1},\lambda_{2};\lambda_{3};z;\rho)=(1-z)^{-\lambda_{1}}\left( {}^{M}\!F_{p,q}^{(\alpha,\beta)}\left( \lambda_{1},\lambda_{3}-\lambda_{2};\lambda_{3};\frac{z}{z-1};\rho\right) \right).
	\end{align*}
\end{theorem} 
\begin{proof}
	Using equation
	\begin{align*}
		\big( 1-z(1-t)\big)^{-\lambda_{1}}=(1-z)^{-\lambda_{1}}\left( 1+\frac{zt}{1-z}\right)^{-\lambda_{1}}
	\end{align*} and by writing $t\to 1-t$ in \eqref{j1}, we have	
	\begin{align*}
		{}^{M}\!&F_{p,q}^{(\alpha,\beta)}(\lambda_{1},\lambda_{2};\lambda_{3};z;\rho)
		\\
		&=\frac{(1-z)^{-\lambda_{1}}}{B(\lambda_{2},\lambda_{3}-\lambda_{2})}\int_{0}^{1}t^{\lambda_{3}-\lambda_{2}-1}(1-t)^{\lambda_{2}-1}\left( 1-\frac{zt}{z-1}\right)^{-\lambda_{1}}{}_{p}^{\alpha}\!M^{\beta}_{q}\left(\kappa_{1},\ldots,\kappa_{p};\mu_{1},\ldots,\mu_{q};\frac{-\rho}{t(1-t)}\right) dt
		\\
		&=(1-z)^{-\lambda_{1}}\left( {}^{M}\!F_{p}^{(\alpha,\beta)}\left( \lambda_{1},\lambda_{3}-\lambda_{2};\lambda_{3};\frac{z}{z-1}\right) \right).\qedhere
	\end{align*}
\end{proof}

\begin{theorem}
The M-confluent hypergeometric function has the equation:
	\begin{align*}
		{}^{M}\!\Phi_{p,q}^{(\alpha,\beta)}(\lambda_{2};\lambda_{3};z;\rho)=\exp(z)\left( {}^{M}\!\Phi_{p,q}^{(\alpha,\beta)}(\lambda_{3}-\lambda_{2};\lambda_{3};-z;\rho)\right).
	\end{align*}
\end{theorem} 
\begin{proof}
Using definition of M-confluent hypergeometric function, we have	
	\begin{align*}
		{}^{M}\!\Phi_{p,q}^{(\alpha,\beta)}&(\lambda_{2};\lambda_{3};z;\rho)
		\\
		&=\sum_{n=0}^{\infty}\frac{{}^{M}\!B_{p,q}^{(\alpha,\beta)}(\lambda_{2}+n,\lambda_{3}-\lambda_{2};\rho)}{B(\lambda_{2},\lambda_{3}-\lambda_{2})}\frac{z^{n}}{n!}\nonumber\\
		&=\frac{1}{B(\lambda_{2},\lambda_{3}-\lambda_{2})}\int_{0}^{1}t^{\lambda_{2}-1}(1-t)^{\lambda_{3}-\lambda_{2}-1}\exp(zt)~{}_{p}^{\alpha}\!M^{\beta}_{q}\left(\kappa_{1},\ldots,\kappa_{p};\mu_{1},\ldots,\mu_{q};\frac{-\rho}{t(1-t)}\right)dt.
	\end{align*}	
	By writing $t\to 1-t$ in above equation, we get
	\begin{equation*}
		{}^{M}\!\Phi_{p,q}^{(\alpha,\beta)}(\lambda_{2};\lambda_{3};z;\rho)=\exp(z)\left( {}^{M}\!\Phi_{p,q}^{(\alpha,\beta)}(\lambda_{3}-\lambda_{2};\lambda_{3};-z;\rho)\right).\qedhere
	\end{equation*}
\end{proof}

\section{Conclusions}
Most of the generalized gamma, beta, Gauss hypergeometric and confluent hypergeometric functions in the literature have been observed to be special cases of the M-gamma, M-beta, M-Gauss hypergeometric and M-confluent hypergeometric functions introduced in this paper, such that:
\\
Abubakar \cite{abubakar},
\begin{align*}
	{}^{M}\!B_{p,q}^{(1,1)}(x,y;\rho)&=\frac{\Gamma(\mu_{1})\ldots\Gamma(\mu_{q})}{\Gamma(\kappa_{1})\ldots\Gamma(\kappa_{p})}{}^{\Psi}\!B_{\rho}^{1,1}
	\left[\!\!\begin{array}{c}
		(\kappa_{i},1)_{1,p} \\
		(\mu_{j},1)_{1,q}  
	\end{array} \Big| x,y \right],
	\\
	{}^{M}\!F_{p,q}^{(1,1)}(\lambda_{1},\lambda_{2};\lambda_{3};z;\rho)&=\frac{\Gamma(\mu_{1})\ldots\Gamma(\mu_{q})}{\Gamma(\kappa_{1})\ldots\Gamma(\kappa_{p})}{}^{\Psi}\!F_{\rho}^{1,1}
	\left[\!\!\begin{array}{c}
		(\kappa_{i},1)_{1,p} \\
		(\mu_{j},1)_{1,q}  
	\end{array} \Big| \lambda_{1},\lambda_{2};\lambda_{3};z \right],
	\\
	{}^{M}\!\Phi_{p,q}^{(1,1)}(\lambda_{2};\lambda_{3};z;\rho)&=\frac{\Gamma(\mu_{1})\ldots\Gamma(\mu_{q})}{\Gamma(\kappa_{1})\ldots\Gamma(\kappa_{p})}{}^{\Psi}\!\Phi_{\rho}^{1,1}
	\left[\!\!\begin{array}{c}
		(\kappa_{i},1)_{1,p} \\
		(\mu_{j},1)_{1,q}  
	\end{array} \Big| \lambda_{2};\lambda_{3};z \right].
\end{align*}
Classic functions \cite{andrews,kilbas},
\begin{align*}
	{}^{M}\Gamma_{1,1}^{(1,1)}(1;1;x;0)&=\Gamma(x),
	\\
	{}^{M}\!B_{1,1}^{(1,1)}(1;1;x,y;0)&=B(x,y),
	\\
	{}^{M}\!F_{1,1}^{(1,1)}(1;1;\lambda_{1},\lambda_{2};\lambda_{3};z;0)&={}_{2}F_{1}(\lambda_{1},\lambda_{2};\lambda_{3};z),
	\\
	{}^{M}\!\Phi_{1,1}^{(1,1)}(1;1;\lambda_{2};\lambda_{3};z;0)&=\Phi(\lambda_{2};\lambda_{3};z).
\end{align*}
Ata and K\i ymaz \cite{ata20},
\begin{align*}
	{}^{M}\Gamma_{p,q}^{(1,1)}(x;\rho)&=\frac{\Gamma(\mu_{1})\ldots\Gamma(\mu_{q})}{\Gamma(\kappa_{1})\ldots\Gamma(\kappa_{p})}{}^{\Psi}\Gamma_{\rho}
	\left[\!\!\begin{array}{c}
		(\kappa_{i},1)_{1,p} \\
		(\mu_{j},1)_{1,q}  
	\end{array} \Big| x \right],
	\\
	{}^{M}\!B_{p,q}^{(1,1)}(x,y;\rho)&=\frac{\Gamma(\mu_{1})\ldots\Gamma(\mu_{q})}{\Gamma(\kappa_{1})\ldots\Gamma(\kappa_{p})}{}^{\Psi}\!B_{\rho}
	\left[\!\!\begin{array}{c}
		(\kappa_{i},1)_{1,p} \\
		(\mu_{j},1)_{1,q}  
	\end{array} \Big| x,y \right],
	\\
	{}^{M}\!F_{p,q}^{(1,1)}(\lambda_{1},\lambda_{2};\lambda_{3};z;\rho)&=\frac{\Gamma(\mu_{1})\ldots\Gamma(\mu_{q})}{\Gamma(\kappa_{1})\ldots\Gamma(\kappa_{p})}{}^{\Psi}\!F_{\rho}
	\left[\!\!\begin{array}{c}
		(\kappa_{i},1)_{1,p} \\
		(\mu_{j},1)_{1,q}  
	\end{array} \Big| \lambda_{1},\lambda_{2};\lambda_{3};z \right],
	\\
	{}^{M}\!\Phi_{p,q}^{(1,1)}(\lambda_{2};\lambda_{3};z;\rho)&=\frac{\Gamma(\mu_{1})\ldots\Gamma(\mu_{q})}{\Gamma(\kappa_{1})\ldots\Gamma(\kappa_{p})}{}^{\Psi}\!\Phi_{\rho}
	\left[\!\!\begin{array}{c}
		(\kappa_{i},1)_{1,p} \\
		(\mu_{j},1)_{1,q}  
	\end{array} \Big| \lambda_{2};\lambda_{3};z \right].
\end{align*}
Chaudhry et al. \cite{chaudhry94,chaudhry97,chaudhry2004},
\begin{align*}
	{}^{M}\Gamma_{1,1}^{(1,1)}(1;1;x;\rho)&=\Gamma_{\rho}(x),
	\\
	{}^{M}\!B_{1,1}^{(1,1)}(1;1;x,y;\rho)&=B_{\rho}(x,y),
	\\
	{}^{M}\!F_{1,1}^{(1,1)}(1;1;\lambda_{1},\lambda_{2};\lambda_{3};z;\rho)&=F_{\rho}(\lambda_{1},\lambda_{2};\lambda_{3};z),
	\\
	{}^{M}\!\Phi_{1,1}^{(1,1)}(1;1;\lambda_{2};\lambda_{3};z;\rho)&=\Phi_{\rho}(\lambda_{2};\lambda_{3};z).
\end{align*}
Kulip et al. \cite{kulip},
\begin{align*}
	{}^{M}\Gamma_{1,2}^{(\alpha,\beta)}(\gamma;\delta,1;x;\rho)&={}^{W}\Gamma_{\rho}^{(\alpha,\beta;\gamma,\delta)}(x),
	\\
	{}^{M}\!B_{1,2}^{(\alpha,\beta)}(\gamma;\delta,1;x,y;\rho)&={}^{W}\!B_{\rho}^{(\alpha,\beta;\gamma,\delta)}(x,y).
\end{align*}
Lee et al. \cite{lee},
\begin{align*}
	{}^{M}\!B_{1,1}^{(1,1)}(1;1;x,y;\rho)&=B(x,y;\rho;1),
	\\
	{}^{M}\!F_{1,1}^{(1,1)}(1;1;\lambda_{1},\lambda_{2};\lambda_{3};z;\rho)&=F_{\rho}(\lambda_{1},\lambda_{2};\lambda_{3};z;1),
	\\
	{}^{M}\!\Phi_{1,1}^{(1,1)}(1;1;\lambda_{2};\lambda_{3};z;\rho)&=\Phi_{\rho}(\lambda_{2};\lambda_{3};z;1).
\end{align*}
Özergin et al. \cite{ozergin},
\begin{align*}
	{}^{M}\Gamma_{1,1}^{(1,1)}(\alpha;\beta;x;\rho)&=\Gamma_{\rho}^{(\alpha,\beta)}(x),
	\\
	{}^{M}\!B_{1,1}^{(1,1)}(\alpha;\beta;x,y;\rho)&=B_{\rho}^{(\alpha,\beta)}(x,y),
	\\
	{}^{M}\!F_{1,1}^{(1,1)}(\alpha;\beta;\lambda_{1},\lambda_{2};\lambda_{3};z;\rho)&=F_{\rho}^{(\alpha,\beta)}(\lambda_{1},\lambda_{2};\lambda_{3};z),
	\\
	{}^{M}\!\Phi_{1,1}^{(1,1)}(\alpha;\beta;\lambda_{2};\lambda_{3};z;\rho)&=\Phi_{\rho}^{(\alpha,\beta)}(\lambda_{2};\lambda_{3};z).
\end{align*}
Parmar \cite{parmar},
\begin{align*}
	{}^{M}\Gamma_{1,1}^{(1,1)}(\alpha;\beta;x;\rho)&=\Gamma_{\rho}^{(\alpha,\beta;1)}(x),
	\\
	{}^{M}\!B_{1,1}^{(1,1)}(\alpha;\beta;x,y;\rho)&=B_{\rho}^{(\alpha,\beta;1)}(x,y),
	\\
	{}^{M}\!F_{1,1}^{(1,1)}(\alpha;\beta;\lambda_{1},\lambda_{2};\lambda_{3};z;\rho)&=F_{\rho}^{(\alpha,\beta;1)}(\lambda_{1},\lambda_{2};\lambda_{3};z),
	\\
	{}^{M}\!\Phi_{1,1}^{(1,1)}(\alpha;\beta;\lambda_{2};\lambda_{3};z;\rho)&=\Phi_{\rho}^{(\alpha,\beta;1)}(\lambda_{2};\lambda_{3};z).
\end{align*}
Rahman et al. \cite{rahman},
\begin{align*}
	{}^{M}\!B_{1,1}^{(\alpha,1)}(1;1;x,y;\rho)&=B_{\rho}^{\alpha;1}(x,y),
	\\
	{}^{M}\!F_{1,1}^{(\alpha,1)}(1;1;\lambda_{1},\lambda_{2};\lambda_{3};z;\rho)&=F_{\rho}^{\alpha;1}(\lambda_{1},\lambda_{2};\lambda_{3};z),
	\\
	{}^{M}\!\Phi_{1,1}^{(\alpha,1)}(1;1;\lambda_{2};\lambda_{3};z;\rho)&=\Phi_{\rho}^{\alpha;1}(\lambda_{2};\lambda_{3};z).
\end{align*}
Sadab et al. \cite{shadab},
\begin{align*}
	{}^{M}\!B_{1,1}^{(\alpha,1)}(1;1;x,y;\rho)&=B_{\alpha}^{\rho}(x,y),
	\\
	{}^{M}\!F_{1,1}^{(\alpha,1)}(1;1;\lambda_{1},\lambda_{2};\lambda_{3};z;\rho)&=F_{\rho,\alpha}(\lambda_{1},\lambda_{2};\lambda_{3};z),
	\\
	{}^{M}\!\Phi_{1,1}^{(\alpha,1)}(1;1;\lambda_{2};\lambda_{3};z;\rho)&=\Phi_{\rho,\alpha}(\lambda_{2};\lambda_{3};z).
\end{align*}
Srivastava et al. \cite{srivastava},
\begin{align*}
	{}^{M}\!B_{1,1}^{(1,1)}(\alpha;\beta;x,y;\rho)&=B_{\rho}^{(\alpha,\beta;1,1)}(x,y),
	\\
	{}^{M}\!F_{1,1}^{(1,1)}(\alpha;\beta;\lambda_{1},\lambda_{2};\lambda_{3};z;\rho)&=F_{\rho}^{(\alpha,\beta;1,1)}(\lambda_{1},\lambda_{2};\lambda_{3};z).
\end{align*}

\end{document}